\documentclass[11pt]{amsart}
\usepackage[margin=1in]{geometry} 
\usepackage[utf8]{inputenc}
\usepackage[T1]{fontenc}
\usepackage{amsmath,amsthm,amssymb}
\usepackage{mathtools,stmaryrd}
\usepackage[english]{babel}
\DeclarePairedDelimiterX{\tbr}[1]{\llbracket}{\rrbracket}{\mkern-5mu\delimsize\llbracket \mathopen{}#1 \mathclose{} \delimsize\rrbracket\mkern-5mu}
\usepackage{hyperref}
\hypersetup{colorlinks=false, linkcolor=blue}
\usepackage{cleveref}
\usepackage{dsfont}
\usepackage{tikz-cd}
\usepackage[table]{xcolor}
\definecolor{knowncell}{gray}{0.92}
\definecolor{priorcell}{RGB}{214,234,248}
\definecolor{newcell}{RGB}{212,239,223}

\usepackage{array}
\usepackage{tikz}
\usetikzlibrary{patterns,patterns.meta}
\newcolumntype{C}[1]{>{\centering\arraybackslash}m{#1}}

\usepackage{calligra,mathrsfs}
\usepackage[shortlabels]{enumitem}
\usepackage{array}
\usepackage{setspace}

\usepackage{booktabs, mathtools}

\usepackage[%
    backend = biber,
    style = alphabetic,
    citestyle = alphabetic,
    maxalphanames = 99,
    maxnames = 99,
    backref = true,
    isbn = false,
    doi = false,
    url = false,
    ]{biblatex}
\newbibmacro{string+url}[1]{
    \iffieldundef{url}{
          #1
    }{
      \href{\thefield{url}}{#1}
    }
}

\DeclareFieldFormat
    [article,misc,inbook,incollection,inproceedings,patent,thesis,unpublished]
  {title}{\usebibmacro{string+url}{\mkbibquote{#1\isdot}}}

\newcommand{\Z}{\mathbb Z}
\newcommand{\N}{\mathbb N}

\renewcommand{\Im}{\operatorname{Im}}

\DeclareMathOperator{\codim}{codim}
\DeclareMathOperator{\str}{str}
\DeclareMathOperator{\prk}{prk}

\theoremstyle{plain}

\newtheorem{theorem}{Theorem}[section]
\newtheorem{prop}[theorem]{Proposition}
\newtheorem{lemma}[theorem]{Lemma}
\newtheorem{claim}[theorem]{Claim}
\newtheorem{cor}[theorem]{Corollary}
\newtheorem*{cor*}{Corollary}

\theoremstyle{definition}
\newtheorem{defn}[theorem]{Definition}
\newtheorem{example}[theorem]{Example}
\newtheorem{remark}[theorem]{Remark}

\crefname{prop}{proposition}{propositions}
\crefname{claim}{claim}{claims}
\Crefname{claim}{Claim}{Claims}
\Crefname{prop}{Proposition}{Propositions}
\crefname{defn}{definition}{definitions}
\Crefname{defn}{Definition}{Definitions}
\crefname{cor}{corollary}{corollaries}
\Crefname{cor}{Corollary}{Corollaries}
\crefname{conj}{conjecture}{conjectures}
\Crefname{conj}{Conjecture}{Conjectures}

\begin{document}
\title{Closedness of bounded strength and partition rank loci}
\author{Benjamin Baily}
\author{Alessandro Danelon}
\thanks{The first author was supported by NSF grant DMS-2101075 and a Simons dissertation fellowship.}
\begin{abstract}
    We show the existence of a quartic form of strength $3$ in the closure of the strength at most $2$ locus, answering a question of Ballico--Bik--Oneto--Ventura in the negative. 
    In the analogous setting of partition rank as introduced by Naslund, we further show that the locus of tensors of partition rank at most $2$ is closed, but the locus of tensors of partition rank at most $3$ is not. Together with a suspension lemma, this answers for all $r$ whether the locus of forms/tensors of strength/partition rank at most $r$ is Zariski-closed.
    \end{abstract}
\maketitle

\section{Introduction}

Consider a homogeneous degree-$d$ form $f$ in the polynomial ring $R = k[x_1,\dots, x_n]$, where $k$ is an algebraically closed field of characteristic zero.
The complexity of $f$ can be quantified as the smallest $r$ such that $f = g_1h_1+\dots + g_rh_r$ for some homogeneous forms $g_1, \dots, g_r, h_1, \dots, h_r$ of strictly-positive degree.
The number $r$ is known as the \textit{strength} of $f$, denoted by $\str(f)$, and it represents one of the many ways to decompose polynomials, and, more generally, tensors. Schmidt introduced strength \cite{schmidt_density_1985} as an auxiliary quantity used to prove uniform bounds on the density of integer points of varieties, and strength appeared again \cite{ah_stillman_2020} in Ananyan--Hochster's proof of Stillman's conjecture. As an object of study in its own right, strength has been studied in combinatorics and algebraic geometry (\cite{gt_distribution_2009}, \cite{kz_properties_2020}, \cite{ess_hartshorne_2021}, \cite{er_transfer_2024}). 

Despite being omnipresent in mathematics, the geometry of tensor decomposition is still little understood.
One natural question is whether the locus of tensors of bounded complexity is defined by polynomial equations, namely, if it is Zariski-closed.
The present paper answers this question in all ranks for two specific complexity measures.
The first is the strength of polynomials; the second is the partition rank of multilinear forms.
Our article completes a work by Ballico--Bik--Oneto--Ventura \cite{ballico-bik-oneto-ventura}, and we summarize in the following table the state of the art for closure conditions for strength and partition rank.

\begin{center}
\begin{tabular}{c|C{2.3cm}|C{2.3cm}|C{2.3cm}|C{2.3cm}}
$r$ & $1$ & $2$ & $3$ & $\geq 4$\\[0.3em]
\hline
Locus of strength at most $r$
  & \cellcolor{knowncell} Closed
  & \cellcolor{newcell} Not closed
  \Cref{thm:str-2-not-closed}
  & \cellcolor{priorcell} Not closed \cite{ballico-bik-oneto-ventura}
  & \cellcolor{newcell} Not closed 
  \Cref{cor:str-higher}
  \\[0.3em]\hline
Locus of partition rank at most $r$
  & \cellcolor{knowncell} Closed
  & \cellcolor{newcell} Closed (!)
  \Cref{thm:prk_2_closed}
  & \cellcolor{newcell} Not closed \Cref{thm:prk_3_4}
  & \cellcolor{newcell} Not closed
  \Cref{cor:prk-higher}
  \\[0.3em]
\end{tabular}
\end{center}

Geometric rank \cite{kopparty-moshkovitz-zuiddam:geometricrank, geng-landsberg:geomrank}, slice rank, local rank (\cite{moshkovitz-zhu:quasilinear}), the $G$-stable rank \cite{derksen:gstable}, commutative and non-commutative rank, analytic rank, Birch rank are known to be closed, while strength, partition rank, tensor rank, subrank, Chow rank, Waring rank/symmetric rank are not.
The cause of closedness/non-closedness seems to depend on whether the rank is defined by notions depending on the geometry of the ambient space (codimension of some space, hyperplanes, etc) or by decomposition/factorization into smaller tensors.

\subsection{Structure of the paper}
The work \cite{ballico-bik-oneto-ventura} of Ballico--Bik--Oneto--Ventura proves that the locus of $d$-forms of bounded strength is not closed in general.
Using the theory of polynomial functors as described in \cite{bik-draisma-snowden:zerochar}, the authors showed the existence of a quartic form of strength $4$ sitting in the closure of forms with strength at most $3$.
However, they wondered in \cite[Question~5]{ballico-bik-oneto-ventura} if the locus of degree-$d$ homogeneous forms of strength at most two could still be closed for every $d$.
Building on the polynomial-functor approach of \cite{ballico-bik-oneto-ventura}, but working directly in an inverse limit of polynomial rings, we give a negative answer to Question~5 for $d = 4$ in Section~\ref{section:strength}.
We briefly discuss the idea of the proof below.

Consider the form
\begin{equation}\label{eq:defH}
    H \coloneqq x^2u + xyv + (y^2q + p(p-yw))
\end{equation}
with $F \coloneqq xy +tp$ and $G\coloneqq x^2 + txw+t^2q$.
Here the variables $x,y,w$ and $p,q,u,v$ are placeholders for, respectively, linear and quadratic forms.
The quartic $H$ was chosen to be a limit of quartics of strength at most $2$:
\begin{equation*}
    H = \lim_{t\to0}\frac{1}{t^2} \left(F\left(t^2v + tp - tyw - xy\right) + G\left(t^2u + y^2\right)\right).
\end{equation*}

We then prove that $H$ has strength $3$ as a polynomial in $x,y,w,p,q,u,v$, where $x,y,w$ have degree one and $p,q,u,v$ have degree two.
To pass from this formal statement to an actual quartic form in finitely many variables, we work in the graded inverse limit $R$ of the polynomial rings $k[x_1, \dots, x_n]$, whose degree-$d$ part consists of the (possibly infinite) formal sums of degree-$d$ monomials in the variables $x_1, x_2, \dots$.
By a theorem of Erman--Sam--Snowden \cite{erman_big_2019}, the ring $R$ is itself a polynomial ring.
Moreover, any tuple of forms of infinite collective strength is part of a system of variables of $R$.
Examples of such tuples are linearly independent linear forms together with the quadrics $p = x_4^2 + x_8^2 + \cdots$, $q = x_5^2 + x_9^2 + \cdots$, $u = x_6^2 + x_{10}^2 + \cdots$ and $v = x_7^2 + x_{11}^2 + \cdots$.
Setting all the other variables of $R$ to zero, a decomposition of $H$ of strength $2$ in $R$ would descend to a decomposition of strength $2$ in the variables $x,y,w,p,q,u,v$, which is impossible.
Finally, a result of Bik \cite{bik} on images of polynomial transformations implies that an element of $R$ has strength at most $2$ if and only if all its truncations to finitely many variables do.
Hence, for $n$ large enough, the truncation of $H$ to $k[x_1, \dots, x_n]$ is a quartic form of strength $3$, and it is still a limit of forms of strength at most $2$.
In contrast to \cite{ballico-bik-oneto-ventura}, this approach does not require describing explicitly the polynomial transformations through which a decomposition of strength $2$ would have to factor.
We collect the necessary background on inverse limits of polynomial rings in Section~\ref{section:inverse-limits}.

In Section~\ref{section:partitionrank}, we formulate and answer the analogous question for the partition rank of a tensor.
A tensor $T$ in $(k^n)^{\otimes d}$ naturally defines a multilinear form $T: (k^n)^{\oplus d}\to k$ corresponding to an element of the multigraded ring
\[
R^{(d)}_n\coloneqq k[\{x_i^{(j)}\}_{i\in [n], j\in [d]}, \deg(x_i^{(j)}) = e_j\in \Z_{\geq 0}^d]
\]
with degree $(1,\dots, 1)$.
We can think of $x_i^{(j)}$ as the dual basis vector $x_i$ of the $i$-th basis vector $v_i$ of the $j$-copy of $k^n$ decorated with the grading $e_j = (0, \dots, 0, 1, 0, \dots, 0)\in \mathbb{Z}^{d}$ with a $1$ in the $j$-th entry.
\begin{defn}[\cite{naslund}]
    The tensor $T$ has \textit{partition rank $r$} if $r$ is the minimum integer for which there exist elements $A_i, B_i\in R^{(d)}_n$ of nonzero multidegree $\underline{a}_i, \underline{b}_i \in \{0, 1\}^d$ with $\underline{a}_i + \underline{b}_i = (1, \dots, 1)$ for $i = 1, \dots, r$, such that $T = A_1B_1 + \dots + A_rB_r$.
\end{defn}
\begin{example}
    Let $V$ be a $2$-dimensional vector space over $k$ with basis $\{v_1, v_2\}$. 
    Let $\{x_1, x_2\}$ be the corresponding basis of $V^*$.
    Consider the tensor $T = v_1\otimes v_1 + 2 v_1 \otimes v_2 - 3 v_2 \otimes v_2$.
    Then $T$ corresponds to the form sending $(v_1, v_1) \mapsto 1, (v_1, v_2) \mapsto 2, (v_2, v_1) \mapsto -3, (v_2, v_2) \mapsto 0$, or equivalently, to the form $x_1^{(1)}x_1^{(2)} + 2 x_1^{(1)}x_2^{(2)} - 3 x_2^{(1)} x_1^{(2)}$.
    Note that $T$ has partition rank at most two as $T = x_1^{(1)}(x_1^{(2)} + 2 x_2^{(2)}) - 3 x_2^{(1)} x_1^{(2)}$.
\end{example}

In the setting of tensors and partition rank, we contribute two facts. The first is that the locus of partition-rank-at-most-two tensors is, in contrast to the case of strength, closed. We accomplish this by first supposing that an element $T$ of degree $(1,\dots, 1)$ can be written as
\[
T = \lim_{t\to 0}\frac{1}{t^b}A(t)B(t) + C(t)D(t)
\]
with $b\geq 0$ and $A(t),\dots, D(t)$ tensors in $ R_n^{(d)}[t]$ and (without loss of generality) $A(t), C(t)$ having no common factor. As a consequence, $T$ belongs to the ideal $I_0\coloneqq \lim_{t\to 0} (A(t), B(t))$, an initial ideal of a complete intersection ideal of height $2$. We show that the set of forms in $I_0$ of degree at most $(1,\dots, 1)$ is contained in a $2$-generated homogeneous ideal, from which it follows that $\prk(T)\leq 2$. Initial ideals of height $2$ complete intersections have been thoroughly studied \cite{ghk_1_2006,ghk_2_2007,ghk_3_2008}; while the authors do not directly use any results from op. cit., we note that both our positive result on the closure of the $\prk\leq 2$ locus and our negative result on the closure of the $\str\leq 2$ locus were inspired by a study of Goto--Heizner--Kim's work.

The second contribution is that the locus of partition-rank-at-most-three tensors is not closed in general.
We do this by showing the existence of a tensor of partition rank $4$ sitting in the closure of the locus of tensors of partition rank at most $3$.
The approach is the same as for strength.
We summarize our approach below.

We consider the polynomial: 
\begin{align*}
    T\coloneqq &p_{3,4}(x_1y_2-x_2y_1) + p_{2,4}(x_3y_1-x_1y_3) + p_{1,4}(x_2y_3-x_3y_2)\\ + &y_4(x_1p_{2,3} + x_2p_{1,3}+ x_3p_{1,2}) + x_4(y_1p_{2,3} + y_2p_{1,3} + y_3p_{1,2}),
\end{align*}
where the variables $x_i, y_i$ have grading $e_i$, and the variable $p_{i,j}$ have grading $e_i + e_j$.
These are placeholders for, respectively, 1-way and 2-way tensors.
We warn the reader that this is typically a confusing point: sometimes we look at $T$ as a polynomial in the variables $x_i, y_i, p_{i,j}$, and some other times we look at $T$ as a tensor given by the product of the tensors $x_i, y_i, p_{i,j}$.
This is due to an intrinsic property of tensors: they can be viewed as polynomials evaluated at some collection of ``algebraically independent'' tensors.
This fact follows from the fact that the graded inverse limit $R^{(d)}$ of the rings $R_n^{(d)}$ is itself a polynomial ring, a result on the same line of \cite{erman_big_2019}.
However, we will try to be clear on which point of view we take every time.

For any choice of tensors $x_i, y_i, p_{i,j}$, the tensor $T$ sits in the closure of the locus of partition-rank-at-most-3 tensors, as it is the following limit
\begin{align*}
T = \lim_{t\to 0} \left(
\begin{aligned}
  &A_{12}p_{34}+A_{13}p_{24}+A_{23}p_{14} \\
  &+ \frac{y_4}{t}(x_3A_{12}+x_2A_{13}+x_1A_{23}) + \frac{x_4}{t}(y_3A_{12}+y_2A_{13}+y_1A_{23})
\end{aligned}
\right),
\end{align*}
with
    \[
    A_{12}\coloneqq x_1y_2 - x_2y_1 + tp_{12},\quad A_{13}\coloneqq x_3y_1 - x_1y_3 + tp_{13}, \quad A_{23}\coloneqq x_2y_3 - x_3y_2 + tp_{23}.
    \]

We then show that $T$ has no decomposition of partition rank $3$ in the inverse-limit polynomial ring $R^{(4)}$ and, exactly as for strength, we conclude that some truncation of $T$ to finitely many variables has partition rank $4$.

Finally, we prove the ``suspension'' lemma for forms and tensors: adding to a form of strength (or partition rank) $k$ a suitable form of strength (partition rank) $1$ in disjoint variables gets us a form of strength/partition rank $k+1$.
Combined with the inverse limit technique, this allows us to answer, for all $r\geq 1$, whether the locus of forms/tensors of strength/partition rank $\leq r$ is closed.

\subsection{AI use disclosure} The authors used Claude Opus 5.5 (Anthropic) to improve the organization and clarity of the paper, including some of the wording in the text, and for computer-algebra checks. The ideas and arguments are the authors’ own; the authors independently reviewed all proofs and take full responsibility for the content.

\section{Inverse limits of polynomial rings}\label{section:inverse-limits}

In this section we develop the background needed for our proof technique, which we use both for strength (\Cref{section:strength}) and for partition rank (\Cref{section:partitionrank}).
The argument always follows the same three steps.
\begin{enumerate}
    \item We write down a \textit{formal} form (or tensor) $f$, namely a polynomial in a few graded variables that act as placeholders for forms (or tensors) of low degree.
    The element $f$ is a limit of elements of strength (or partition rank) at most $r$, but it does not admit a decomposition of strength (or partition rank) at most $r$ in the placeholder variables.
    \item We realize the placeholders as elements of the graded inverse limit of the polynomial rings in $n$ variables.
    By a result of Erman--Sam--Snowden \cite{erman_big_2019} this inverse limit is again a polynomial ring, and the placeholders can be chosen to be part of a system of variables.
    It follows that $f$ has no decomposition of strength (or partition rank) at most $r$ in the inverse limit either.
    \item A result of Bik on images of inverse limits of polynomial transformations \cite{bik} then produces a truncation of $f$ to finitely many variables with no such decomposition.
\end{enumerate}
For the convenience of the reader we treat the symmetric case (forms and strength) and the multigraded case (tensors and partition rank) in parallel, as the symmetric case might feel more natural at first.

\subsection{Polynomial functors and their limits}
Let $S^d$ denote the $d$-th symmetric power functor, i.e., for a vector space $V$ the space $S^d(V)$ contains the symmetric tensors of $V^{\otimes d}$.
For our purposes it is convenient to identify an element of $S^d(V)$ with a homogeneous degree-$d$ polynomial in variables $x_1, \dots, \allowbreak x_{\dim(V)}$.
For any linear map $\phi: V \to W$, the restriction of the Kronecker product $\phi^{\otimes d}: V^{\otimes d} \to W^{\otimes d}$ to the symmetric subspace gives the corresponding linear map $S^d_{\phi}: S^d(V) \to S^d(W)$.
Similarly, for a finite set $J$ we denote by $T^J$ the functor sending $V$ to $V^{\otimes J}$ and $\phi$ to $\phi^{\otimes J}$.
In characteristic zero, both $S^d$ and $T^J$ are \textit{polynomial functors}, i.e., finite direct sums of Schur functors; we refer to \cite{fulton-harris:reprfirstcourse} for an introduction to Schur functors and to \cite{bik-draisma-snowden:zerochar} for the general theory of polynomial functors.
In this paper we only need finite direct sums of the functors $S^d$ and $T^J$.

A map between vector spaces $f: V \to W$ is \textit{polynomial} if its coordinates with respect to a basis of $W$ are polynomial functions on $V$.

\begin{defn}
    Let $P$ and $Q$ be finite direct sums of the functors $S^d$ and $T^J$, considered as functors into the category of sets.
    A \textit{polynomial transformation} $\alpha: Q \to P$ is a natural transformation between $Q$ and $P$ such that for each finite dimensional vector space $V$ the map $\alpha_V : Q(V) \to P(V)$ is polynomial.
\end{defn}

\begin{example}
    Consider the polynomial transformation $\alpha: (S^1)^{\oplus 2} \oplus S^2 \to S^4$ sending $(X_1, X_2, X_3)$ to $X_1^2X_3 -X_2^4$.
    Then for every finite dimensional vector space $V$, $\alpha_V$ sends $(l_1, l_2, q) \in S^1(V)^{\oplus 2} \oplus S^2(V)$ to $l_1^2q - l_2^4 \in S^4(V)$.
\end{example}

The polynomial transformations relevant to us are given by multiplication.
For a tuple $\underline{a} = (a_1, \dots, a_r)$ of integers with $1 \leq a_i \leq d-1$, the map
\[
\sigma_{\underline{a}}: \bigoplus_{i=1}^r \left(S^{a_i}\oplus S^{d-a_i}\right) \to S^d, \qquad (g_1,h_1,\dots,g_r,h_r) \mapsto g_1h_1+\dots+g_rh_r
\]
is a polynomial transformation.
Allowing some of the $g_i, h_i$ to be zero, a form in $S^d(V)$ has strength at most $r$ if and only if it lies in the image of $\sigma_{\underline{a},V}$ for one of the finitely many tuples $\underline{a}$.
The analogous polynomial transformations $\alpha_{I_1,\dots,I_r}$ for partition rank are defined in \Cref{section:partitionrank}.

Consider a functor $P$ as above, and evaluate it on $k^n$ for every $n$.
The projections $\pi_n : k^n \to k^{n-1}$ onto the first $n-1$ coordinates make $(P(k^n), P_{\pi_n})_n$ an inverse system, and we denote by $P_\infty$ its inverse limit.
Given a point $p \in P_\infty$, we denote by $p_n$ its projection to $P(k^n)$, so that $p = (p_n)_{n \in \mathbb{N}}$.
The space $P_\infty$ is (typically) an infinite-dimensional algebraic variety.

The elements of $S^d_\infty$ are the (possibly infinite) formal sums of homogeneous degree-$d$ monomials in the variables $x_1, x_2, \dots$.
For example, the element $x_1^2 + x_2^2 + \cdots$ sits in $S^2_\infty$.
As a non-symmetric example, the inverse limit of the second tensor power $T^{[2]}$ is the space of $\mathbb{N}\times \mathbb{N}$-matrices with possibly infinitely many nonzero entries.
Note that new points arise from the limiting procedure, namely the ones that do not come from any $P(k^n)$: in the example of matrices, the infinite identity matrix is one of these points.

By naturality, a polynomial transformation $\alpha: Q \to P$ induces a map $\alpha_\infty : Q_\infty \to P_\infty$ given by $\alpha_\infty (q) = (\alpha_{k^n} (q_n))_{n\in \mathbb{N}}$.
We need the following compactness property of these maps.

\begin{lemma}\label{lemma:belonging}
    Let $\alpha: Q \to P$ be a polynomial transformation and let $f \in P_\infty$.
    Then $f \in \Im(\alpha_\infty)$ if and only if $f_n \in \Im(\alpha_{k^n})$ for every $n$.
\end{lemma}

\begin{proof}
    The ``only if'' direction follows from the definition of $\alpha_\infty$.
    The ``if'' direction is \cite[Lemma~4.5.24]{bik}.
\end{proof}

\subsection{The inverse limit rings}
Let $R_n \coloneqq k[x_1, \dots, x_n]$ with the standard grading, and let $R$ be the inverse limit of the diagram
\[
\dots \to R_n \to R_{n-1} \to \dots \to R_1,
\]
where $R_n \to R_{n-1}$ sets $x_n$ to zero, in the category of graded rings.
Concretely, $R = \bigoplus_{d \geq 0} R_d$ where $R_d = \varprojlim_n (R_n)_d = S^d_\infty$ is the space of formal $k$-linear combinations of degree-$d$ monomials in the variables $x_1, x_2, \dots$.
We denote by $\pi_n: R \to R_n$ the natural projections, which are homomorphisms of graded rings, and for $f \in R$ we write $f_n \coloneqq \pi_n(f)$.

In the multigraded setting we let $R_n^{(d)} \coloneqq k[x_i^{(j)} \mid i\in [n], j\in [d]]$ with the $\N^d$-grading given by $\deg(x^{(j)}_i) = e_j$, and we let $R^{(d)}$ denote the inverse limit of the diagram
\[
\dots \to R^{(d)}_n\to R^{(d)}_{n-1}\to \dots \to R_1^{(d)},
\]
where $R^{(d)}_n\to R^{(d)}_{n-1}$ sets $x_n^{(1)}, \dots, x_n^{(d)}$ to zero, in the category of $\N^d$-graded rings.
Again, $\pi_n : R^{(d)} \to R^{(d)}_n$ denote the projections, and $f_n \coloneqq \pi_n(f)$.
For $v = (v_1,\dots, v_d)\in \N^d$ we set $|v| = v_1+\dots+v_d$, and we write $\mathbf{1} = e_1 + \dots + e_d$.

Let $I$ be the ideal of $R$ (resp.\ $R^{(d)}$) of elements with zero constant term.
\begin{theorem}[{\cite[Theorems~1.1 and~1.2]{erman_big_2019}}]\label{thm:ESS}
    Let $E \subseteq R$ be a set of homogeneous elements of positive degree whose images in $I/I^2$ form a $k$-basis.
    Then the $k$-algebra homomorphism $k[Y_e \mid e \in E] \to R$ sending $Y_e$ to $e$ is an isomorphism.
    In particular, $R$ is a polynomial ring.
\end{theorem}

\begin{lemma}\label{lem:R_inf}
    The ring $R^{(d)}$, graded by total degree $|\cdot|$, is a polynomial ring in homogeneous elements.
\end{lemma}
\begin{proof}
    As a graded ring, $R^{(d)}$ can be described by describing each graded piece: for each $v\in \N^d$, $R^{(d)}_v$ is the set of formal $k$-linear combinations of monomials of degree $v$ in the variables $\{x_i^{(j)}\}$, with the maps $R^{(d)}_v\to (R_n^{(d)})_v$ given by the natural projection.

    Let $S$ denote the inverse limit $\varprojlim_n R_n^{(d)}$ in the category of $\N$-graded rings, where we set $\deg(x_i^{(j)}) = 1$.
    Ordering the variables as $x_1^{(1)}, \dots, x_1^{(d)}, x_2^{(1)}, \dots, x_2^{(d)}, \dots$, the system $(R_n^{(d)})_n$ is a cofinal subsystem of the system of standard-graded polynomial rings defining $R$, so $S \cong R$.
    For any $t\in \N$ the vector space $S_t$ is the set of formal $k$-linear combinations of monomials of degree $t$ in the $\{x_i^{(j)}\}$; see \cite[\S1.1]{erman_big_2019}.
    As the set of vectors $v\in \N^d$ with $|v|=t$ is finite, it follows that
    \[
    S_t \cong \bigoplus_{v\in \N^d, |v|=t}R^{(d)}_v,
    \]
    hence $S\cong R^{(d)}$ as graded rings, and we conclude by \Cref{thm:ESS}.
\end{proof}

Under the identifications above, $S^a_\infty = R_a$ and, for $J\subseteq [d]$, $(T^J)_\infty = R^{(d)}_{e_J}$ where $e_J \coloneqq \sum_{j \in J} e_j$.
In particular, the inverse limit of the $d$-th tensor power is $R^{(d)}_{\mathbf{1}}$.
Since the projections $\pi_n$ are ring homomorphisms, the map $\sigma_{\underline{a},\infty}$ sends $(g_1, h_1, \dots, g_r, h_r) \in \bigoplus_i (R_{a_i} \oplus R_{d-a_i})$ to the product $g_1h_1 + \dots + g_rh_r$ computed in $R$, and similarly for the transformations $\alpha_{I_1,\dots,I_r}$ of \Cref{section:partitionrank} and $R^{(d)}$.
This motivates the following definitions.

\begin{defn}\label{defn:str-prk-in-subrings}
    Let $\mathcal{B}$ be either $R$ or $R^{(d)}$, or a graded subring of one of them.
    \begin{enumerate}
        \item If $\mathcal{B} \subseteq R$ and $f \in \mathcal{B}_d$, the \textit{strength} $\str_{\mathcal{B}}(f)$ of $f$ in $\mathcal{B}$ is the minimal $r$ such that $f = g_1h_1 + \dots + g_rh_r$ with $g_i,h_i\in \mathcal{B}$ homogeneous of positive degree, and $\str_{\mathcal{B}}(f) = \infty$ if no such $r$ exists.
        \item If $\mathcal{B} \subseteq R^{(d)}$, a \textit{tensor} in $\mathcal{B}$ is a homogeneous element of nonzero degree in $\{0,1\}^d$, and for $f \in \mathcal{B}_{\mathbf{1}}$ the \textit{partition rank} $\prk_{\mathcal{B}}(f)$ of $f$ in $\mathcal{B}$ is the minimal $r$ such that $f = A_1B_1 + \dots + A_rB_r$ with $A_i,B_i$ tensors in $\mathcal{B}$ (necessarily of complementary degrees), and $\prk_{\mathcal{B}}(f) = \infty$ if no such $r$ exists.
    \end{enumerate}
    For $\mathcal{B} = R_n$ and $\mathcal{B} = R_n^{(d)}$ these are the usual notions, and we omit the subscript $\mathcal{B}$ when no confusion can arise.
\end{defn}

Taking homogeneous components of a finite sum of products of elements of $I$, one sees that a homogeneous element $f\in R_d$ lies in $I^2$ if and only if $\str_R(f) < \infty$.
Similarly, a homogeneous element of $R^{(d)}$ of degree $v \in \{0,1\}^d$ lies in $I^2$ if and only if it is a finite sum of products of tensors of complementary degrees.

\begin{defn}
    A tuple $(f_1, \dots, f_s)$ of homogeneous forms in $R$ has \textit{infinite collective strength} if any nontrivial linear combination of the $f_i$ of the same degree has infinite strength in $R$.
\end{defn}

By the discussion above, $(f_1, \dots, f_s)$ has infinite collective strength if and only if the images of $f_1, \dots, f_s$ in $I/I^2$ are linearly independent.
Elements of this kind play a structural role.

\begin{defn}
    Homogeneous elements $f_1, \dots, f_s$ of $I \subseteq R$ (resp.\ $I \subseteq R^{(d)}$) are \textit{part of a system of variables} if their images in $I/I^2$ are linearly independent.
    In $R^{(d)}$, homogeneous means homogeneous for the $\N^d$-grading.
\end{defn}

The terminology is justified by the following lemma: elements that are part of a system of variables can be treated as independent variables, and all other variables can be set to zero.

\begin{lemma}\label{lemma:GLgenisvariables}
    Let $f_1, \dots, f_s \in R$ (resp.\ $R^{(d)}$) be homogeneous elements that are part of a system of variables, and let $\mathcal{A} \coloneqq k[f_1, \dots, f_s]$.
    Then $f_1, \dots, f_s$ are algebraically independent, and there exists a homomorphism of graded (resp.\ $\N^d$-graded) $k$-algebras $\pi : R \to \mathcal{A}$ (resp.\ $\pi : R^{(d)} \to \mathcal{A}$) that restricts to the identity on $\mathcal{A}$.
\end{lemma}

\begin{proof}
    As $I/I^2$ is $\N^d$-graded, we can extend $f_1, \dots, f_s$ to a set $E$ of $\N^d$-homogeneous elements whose images in $I/I^2$ form a $k$-basis.
    These elements are also homogeneous for the total degree, so by \Cref{lem:R_inf} and \Cref{thm:ESS} the ring $R^{(d)}$ is the polynomial ring $k[E]$.
    Hence $f_1, \dots, f_s$ are algebraically independent, and sending $e \in E$ to itself if $e \in \{f_1, \dots, f_s\}$ and to $0$ otherwise defines an $\N^d$-graded homomorphism $\pi: R^{(d)} \to \mathcal{A}$ that restricts to the identity on $\mathcal{A}$.
    The case of $R$ is the same.
\end{proof}

Applying \Cref{lemma:GLgenisvariables} to a full basis shows that there exist index sets $\{I_v\}_{v\in \N^d}$ such that $R^{(d)}$ is isomorphic, as an $\N^d$-graded ring, to the polynomial ring $k[\{X_{v,i}: v\in \N^d, i\in I_v\}]$ where $\deg(X_{v,i}) = v$.
For this reason we can think of tensors as polynomials evaluated at some collection of ``algebraically independent'' tensors, and we use this identification extensively in the following sections.

\begin{example}\label{ex:variables}
    In degree one there is nothing to check: $(I^2)_1 = 0$ in $R$ and $(I^2)_{e_j} = 0$ in $R^{(d)}$, so linearly independent elements of these degrees are part of a system of variables.
    Elements in $S^d_\infty$ can have infinite strength even for $d \geq 2$.
    For example, the quadric $x_1^2 + x_2^2 + \cdots \in R_2$ does not lie in $I^2$: if it were equal to $\sum_{i = 1}^s \ell_i\ell'_i$ with $\ell_i, \ell'_i\in R_1$, then each truncation $x_1^2 + \cdots + x_n^2$ would have rank at most $2s$.
    In the same way, the infinite identity matrix $\sum_{m \geq 1} x_m^{(k)}x_m^{(l)} \in R^{(d)}_{e_k + e_l}$ with $k \neq l$ does not lie in $I^2$, as the elements of $(I^2)_{e_k+e_l}$ are matrices of finite rank.
\end{example}

\subsection{The main technique}

\begin{prop}\label{prop:transfer}
    Let $r \geq 0$.
    \begin{enumerate}
        \item If $f \in R_d$ satisfies $\str(f_n) \leq r$ for every $n$, then $\str_R(f) \leq r$.
        \item If $f \in R^{(d)}_{\mathbf{1}}$ satisfies $\prk(f_n) \leq r$ for every $n$, then $\prk_{R^{(d)}}(f) \leq r$.
    \end{enumerate}
\end{prop}

\begin{proof}
    We prove the first statement; the second is analogous, using the polynomial transformations $\alpha_{I_1, \dots, I_r}$ of \Cref{section:partitionrank} in place of the $\sigma_{\underline{a}}$.
    For every $n$, the form $f_n$ lies in the image of $\sigma_{\underline{a}, k^n}$ for one of the finitely many tuples $\underline{a}$.
    If $f_n$ lies in the image of $\sigma_{\underline{a}, k^n}$, then $f_m$ lies in the image of $\sigma_{\underline{a}, k^m}$ for every $m \leq n$, by naturality of $\sigma_{\underline{a}}$ with respect to the projection $k^n \to k^m$.
    Hence there is a single tuple $\underline{a}$ such that $f_n \in \Im(\sigma_{\underline{a}, k^n})$ for every $n$.
    By \Cref{lemma:belonging}, $f$ lies in the image of $\sigma_{\underline{a}, \infty}$, i.e., $f = g_1h_1 + \dots + g_rh_r$ with $g_i \in R_{a_i}$ and $h_i \in R_{d-a_i}$.
    Hence $\str_R(f) \leq r$.
\end{proof}

\begin{cor}\label{cor:main-technique}
    Let $f_1, \dots, f_s \in R$ (resp.\ $R^{(d)}$) be homogeneous elements that are part of a system of variables, and let $\mathcal{A} \coloneqq k[f_1, \dots, f_s]$.
    Let $f \in \mathcal{A}_d$ (resp.\ $f \in \mathcal{A}_{\mathbf{1}}$) with $\str_{\mathcal{A}}(f) > r$ (resp.\ $\prk_{\mathcal{A}}(f) > r$).
    Then there exists $n$ such that $\str(f_m) > r$ (resp.\ $\prk(f_m) > r$) for every $m\geq n$.
\end{cor}

\begin{proof}
    We prove the statement for strength; the proof for partition rank is the same.
    Let $\pi : R \to \mathcal{A}$ be the retraction of \Cref{lemma:GLgenisvariables}.
    If $f = g_1h_1 + \dots + g_rh_r$ with $g_i,h_i \in R$ homogeneous of positive degree, then
    \[
    f = \pi(f) = \pi(g_1)\pi(h_1) + \dots + \pi(g_r)\pi(h_r),
    \]
    where each $\pi(g_i)$ is either zero or homogeneous of degree $\deg(g_i)$, and similarly for $\pi(h_i)$.
    Discarding the vanishing terms gives $\str_{\mathcal{A}}(f) \leq r$, a contradiction.
    Hence $\str_R(f) > r$, and \Cref{prop:transfer} gives some $n$ with $\str(f_n) > r$.
    Finally, if $\str(f_m) \leq r$ for some $m \geq n$, then applying the projection $R_m \to R_n$ to a decomposition of $f_m$ gives $\str(f_n) \leq r$, a contradiction.
\end{proof}

Concretely, $f_n = \pi_n(f)$ is obtained by substituting the truncations $(f_1)_n, \dots, (f_s)_n$ for $f_1, \dots, f_s$ in the polynomial expression of $f$.
Moreover, if $f$ is a limit of elements of $\mathcal{A}$ of strength (resp.\ partition rank) at most $r$, explicitly given by a polynomial family in a parameter $t$, then applying $\pi_n$ to this family shows that $f_n$ lies in the closure of the locus of strength (resp.\ partition rank) at most $r$.

\section{Strength}\label{section:strength}
\begin{lemma}\label{lemma:str3}
    The polynomial $H = x^2u + xyv + y^2q + p(p-yw)$ of \eqref{eq:defH} has strength 3 as an element of the ring $k[x, y, w, p, q, u, v]$ with $x,y,w$ variables of degree one, and $p,q,u,v$ variables of degree $2$. 
\end{lemma}

\begin{proof}
Suppose $H = AF + BG$.
There are three cases for the degrees of $(A,F)$ and $(B,G)$:
\begin{enumerate}[(a)]
    \item \label{case1} $(1,3) $ and $(1,3)$;
    \item \label{case2} $(1, 3)$ and $ (2,2)$;
    \item \label{case3} $(2, 2)$ and $(2, 2)$.
\end{enumerate}
To rule out the cases~\ref{case1}~and~\ref{case2}, we argue that $H$ is irreducible mod $(A,v,w)$, which implies $H\neq AF+BG$. We first note that $H\equiv x^2u + y^2q+p^2 \mod(v,w)$. When $A\in (v,w)$ this is clear. When $A$ is nonzero mod $(v,w)$, after reducing mod $A$ and possibly changing variables the resulting form $x^2u+p^2 +(A,v,w)$ is irreducible.

To rule out the case~\ref{case3}, we argue as in \cite[Theorem 23, case d]{ballico-bik-oneto-ventura}. Write $A = A' + A''$ where $A' \in k[x,y,w]_2, A''\in k[p,q,u,v]_1$ and similarly for $B,F,G$. Expanding $H$ as a polynomial over $k[x,y,w]$ in the variables $p,q,u,v$, the identity $H = AF + BG$ yields the identities
\begin{align}
    A'F' + B'G' &= 0,\nonumber\\
    A'F'' + A''F' + B'G'' + B''G' &= x^2u+xyv + y^2q-ywp,\label{eq:H=AFBG_line-2}\\
    A''F'' + B''G'' &= p^2.\label{eq:H=AFBG_line-3}
\end{align}
\Cref{eq:H=AFBG_line-2} implies that $u,v,p,q$ are contained in the span of $A'',B'',F'',G''$, so $A'',B'',F'',G''$ are linearly independent linear forms in $k[p,q,u,v]_1$. Consequently, $A''F'' + B''G''$ is irreducible, contradicting \Cref{eq:H=AFBG_line-3}.
\end{proof}

We now use the inverse limit ring $R$ of \Cref{section:inverse-limits} to substitute forms in enough variables for the placeholders appearing in $H$, so as to obtain a form of strength $3$ that is a limit of forms of strength at most $2$.

\begin{theorem}\label{thm:str-2-not-closed}
    Consider the quadrics
    \[
    p \coloneqq \sum_{i\geq 1}x_{4i}^2,\quad q \coloneqq \sum_{i\geq 1}x_{4i+1}^2,\quad u \coloneqq \sum_{i\geq 1}x_{4i+2}^2,\quad v \coloneqq \sum_{i\geq 1}x_{4i+3}^2 \quad \in R_2,
    \]
    and let $H \in R_4$ be the form of Equation~\eqref{eq:defH} with $x = x_1$, $y = x_2$ and $w = x_3$.
    Then for $n$ sufficiently large the quartic form
    \[
    H_n = x_1^2u_n + x_1x_2v_n + x_2^2q_n + p_n(p_n-x_2x_3) \in k[x_1, \dots, x_n]
    \]
    has strength $3$ and lies in the closure of the locus of quartic forms of strength at most $2$.
\end{theorem}

\begin{remark}\label{rem:high-strength}
    The explicit choice of forms in \Cref{thm:str-2-not-closed} is made for concreteness.
    In fact, there exists $N$ such that for any linearly independent linear forms $x, y, w$ and any quadrics $p, q, u, v$ of collective strength at least $N$, in any number of variables, the quartic $H$ of Equation~\eqref{eq:defH} has strength $3$.
    Indeed, by universality of high-strength tuples \cite{bik-danelon-draisma-eggermont:univ}, for $N$ large enough any such tuple can be sent by a linear substitution of the variables to the tuple $(x_1, x_2, x_3, p_n, q_n, u_n, v_n)$ of \Cref{thm:str-2-not-closed}, with $n$ as there.
    Linear substitutions do not increase strength, so $\str(H) \geq \str(H_n) = 3$, while the decomposition $H = x(xu + yv) + y(yq - wp) + p\cdot p$ gives $\str(H) \leq 3$.
\end{remark}

\begin{proof}[Proof of \Cref{thm:str-2-not-closed}]
    We first check that $x_1, x_2, x_3, p, q, u, v$ are part of a system of variables of $R$.
    As $I/I^2$ is graded, it suffices to check this separately in degrees one and two.
    In degree one there is nothing to prove.
    In degree two, a nontrivial linear combination $\lambda_1 p + \lambda_2 q + \lambda_3 u + \lambda_4 v$ has truncations that are diagonal quadrics of unbounded rank, so it does not lie in $I^2$ by the argument of \Cref{ex:variables}.

    By \Cref{lemma:GLgenisvariables}, the subring $\mathcal{A} \coloneqq k[x_1, x_2, x_3, p, q, u, v]$ of $R$ is a polynomial ring in these elements, graded by $\deg(x_i) = 1$ and $\deg(p) = \deg(q) = \deg(u) = \deg(v) = 2$.
    By \Cref{lemma:str3}, $\str_{\mathcal{A}}(H) = 3 > 2$, so \Cref{cor:main-technique} gives $\str(H_n) > 2$ for $n$ sufficiently large.
    On the other hand, the decomposition
    \[
    H = x(xu + yv) + y(yq - wp) + p\cdot p
    \]
    shows that $\str(H_n)\leq 3$ for every $n$.

    To see that $H_n$ is in the closure of the strength-at-most-$2$ locus, set
    \[
    F \coloneqq xy +tp,\qquad G\coloneqq x^2 + txw+t^2q
    \]
    in $\mathcal{A}[t]$. A direct computation shows
    \[
    F\left(t^2v + tp - tyw - xy\right) + G\left(t^2u + y^2\right) = t^2 H + t^3(pv + xwu) + t^4 qu.
    \]
    Applying $\pi_n$, we obtain that for every $t \in k\setminus\{0\}$ the form $H_n + t(p_nv_n + x_1x_3u_n) + t^2 q_nu_n$ has strength at most $2$.
    Hence
    \[
    H_n = \lim_{t\to 0}\frac{1}{t^2}\left(F_n\left(t^2v_n + tp_n - tx_2x_3 - x_1x_2\right) + G_n\left(t^2u_n + x_2^2\right)\right),
    \]
    where $F_n$ and $G_n$ are obtained by applying $\pi_n$ to $F$ and $G$, lies in the closure of the locus of quartic forms of strength at most $2$.
\end{proof}

Finally, we want to show that the locus of forms of strength at most $r$ is never closed for $r \geq 2$ provided that the forms have enough variables.

\begin{lemma}[Suspension for Strength]\label{lemma:str-suspension}
    Let $F$ be a form of degree $d$ in the degree-$1$ variables $x_1,\dots, x_n$. Let $y$ be a variable of degree 1 and let $z$ a variable of degree $d-1$. Then $\str(F+yz) = \str(F)+1$.  
\end{lemma}
\begin{proof}
    The $d=2$ case is immediate by relating the strength of a form $F$ to the rank of the symmetric matrix $M_F$ corresponding to $F$.
    It is easy to check that $\operatorname{str}(F) = \left\lceil \frac{\operatorname{rk}(M_F)}{2} \right\rceil$.
    As $y$ and $z$ are independent variables from the $x_i$, we have $\operatorname{rk}(M_{F + yz}) = \operatorname{rk}(M_F) + 2$.
    Now the statement follows.
    
    Assume $d\geq 3$. 
    The inequality $\str(F+yz)\leq \str(F) + 1$ is immediate. For the reverse inequality, suppose we have
    \begin{equation}\label{eq:suspension-str-decomp}
    F + yz = G_1H_1 + \dots +G_rH_r. 
    \end{equation}
    As the variable $z$ must appear in the right-hand side, we must have (WLOG) $\deg(G_1)=1,\deg(H_1)=d-1$. Write $H_1 = h_1+bz$ with $h_1\in k[x_1,\dots, x_n,y]$, and (again WLOG) $b = 1$. For any $\lambda\in k$ and any $2\leq i\leq r$, we have 
    \[
    G_1H_1 + G_iH_i = (G_1 + \lambda G_i)(h_1 + z) + G_i(H_i - \lambda h_1 - \lambda z),
    \]
    so we may iteratively replace $G_1, (G_2,H_2),\dots, (G_r, H_r)$ until $G_2,\dots, H_r\in k[x_1,\dots, x_n, y]$. As $\deg(G_1) = 1 < d-1$, we have $G_1\in k[x_1,\dots, x_n,y]$; write $G_1 = g_1 + ay$ where $a\in k, g_1\in k[x_1,\dots, x_n]$. Equating the coefficients of $yz$ in both sides of \eqref{eq:suspension-str-decomp} allows us to conclude that $a=1$.
    
    A similar replacement, for $u\in k[x_1,\dots, x_n,y]$, allows us to write 
    \[
    G_1H_1 + G_iH_i =  G_1(H_1 + uH_i) + (G_i - uG_1)H_i,
    \]
    so we may again replace $H_1, (G_2,H_2),\dots, (G_r,H_r)$ until $G_2,\dots, H_r\in k[x_1,\dots, x_n]$. Equating the coefficients of $y^2,\dots, y^d$ in both sides of \eqref{eq:suspension-str-decomp} allows us to conclude that $h_1\in k[x_1,\dots, x_n, z]$. We may therefore rewrite
    \[
    F + yz = (g_1 + y)(h_1+z) + G_2H_2 + \dots + G_rH_r,
    \]
    where $g_1,h_1,G_i,H_i\in k[x_1,\dots, x_n]$. It follows that $g_1 = h_1 = 0$, hence $F = G_2H_2 + \dots + G_rH_r$; we conclude that $\str(F)\leq \str(F+yz)-1$.
\end{proof}

Combining \Cref{lemma:str-suspension} with the technique of \Cref{section:inverse-limits}, we obtain non-closedness for all $r \geq 2$.

\begin{cor}\label{cor:str-higher}
    For every $r \geq 2$ there exist an integer $n$ and a quartic form in $k[x_1, \dots, x_n]$ of strength $r+1$ that lies in the closure of the locus of quartic forms of strength at most $r$.
\end{cor}
\begin{proof}
    We proceed by induction on $r$, the case $r = 2$ being \Cref{thm:str-2-not-closed}.
    Let $F \in k[x_1, \dots, x_n]$ be a quartic form of strength $r+1$ in the closure of the locus of quartics of strength at most $r$, and view $F$ as an element of $R$.
    Set $y \coloneqq x_{n+1}$ and $z \coloneqq \sum_{i \geq n+2} x_i^3 \in R_3$.
    We claim that $x_1, \dots, x_{n+1}, z$ are part of a system of variables, and only need to show $z\notin I^2$.
    If $z = \sum_{j=1}^s g_jh_j$ with $g_j, h_j \in R$ homogeneous of positive degree, then $\str(z_m) \leq s$ for every $m$.
    However, after extending scalars to an algebraic closure of $k$, the cubic $z_m = x_{n+2}^3 + \dots + x_m^3$ defines a smooth hypersurface in the variables $x_{n+2}, \dots, x_m$, while a cubic of strength at most $s$ has a singular locus of codimension at most $2s$; hence $\str(z_m)$ is unbounded.

    By \Cref{lemma:GLgenisvariables}, $\mathcal{A} \coloneqq k[x_1, \dots, x_{n+1}, z]$ is a polynomial ring with $\deg(z) = 3$, and \Cref{lemma:str-suspension} gives $\str_{\mathcal{A}}(F + yz) = r+2$.
    By \Cref{cor:main-technique}, there exists $m > n$ such that the quartic form $F + x_{n+1}z_m \in k[x_1, \dots, x_m]$ has strength greater than $r+1$, hence equal to $r+2$.
    Finally, the map $S^4(k^n) \to S^4(k^m)$ sending $g \mapsto g + x_{n+1}z_m$ is continuous and sends forms of strength at most $r$ to forms of strength at most $r+1$.
    As $F$ lies in the closure of the former set, $F + x_{n+1}z_m$ lies in the closure of the latter.
\end{proof}

\section{Partition Rank}\label{section:partitionrank}
Let $V$ be a vector space over the field $k$ of dimension $n$.
An element $T $ in the tensor power $ V^{\otimes d}$ can be interpreted as a multilinear form on the space $(V^*)^{\times d}$.
Let $R_n^{(d)} $ be the polynomial ring $k[\{x_{i}^{(j)}\}_{i\in [n],j\in [d]}]$ with $\mathbb{N}_{\geq 0}^d$-grading given by $\deg(x_i^{(j)}) = e_j$.
The \textit{degree} of a monomial in $R_n^{(d)}$ is the sum of the entries of its $\mathbb{N}_{\geq 0}^d$-grading.
We think of $x_i^{(j)}$ as the dual basis vector $x_i$ of the $i$-th basis vector $v_i$ of the $j$-th copy of $V$ decorated with the grading $e_j = (0, \dots, 0, 1, 0, \dots, 0)\in \mathbb{Z}^{d}$ with a $1$ in the $j$-th entry.

\begin{defn}
     A \textit{tensor} in $R_n^{(d)}$ is a homogeneous element of strictly positive degree with $\mathbb{N}_{\geq 0}^d$-grading in $\{0,1\}^d$.
\end{defn}

Let $J$ be a subset of $[d] = \{1,\dots, d\}$, let $e_J$ denote the vector $\sum_{j \in J}e_j$, and write $R_{e_J}^{(d)}$ for the submodule of $R_n^{(d)}$ of elements with $\mathbb{N}_{\geq 0}^d$-grading $e_J$.
We identify the space $V^{\otimes J}$ with the module $R_{e_J}^{(d)}$.

\begin{example}
    In $R_4^{(4)}$, the elements $x_1^{(1)} + x_2^{(1)}, x_1^{(1)}x_2^{(2)} + x_3^{(1)}x_4^{(2)}$, $x_1^{(1)}x_1^{(2)}x_1^{(3)}x_1^{(4)}$ are tensors, while $x_1^{(1)} + 1, x_1^{(1)} + x_1^{(2)}$, $(x_1^{(1)})^2 x_2^{(2)}$ and $x_1^{(1)}x_2^{(1)}x_3^{(1)}x_4^{(1)}$ are not.
\end{example}

\begin{lemma}\label{lemma:tensor-factor}
    Let $T\in R_n^{(d)}$ be a tensor such that $T = AB$ for some nonconstant elements $A, B$ of $R_n^{(d)}$.
    Then $A,B$ are tensors.
\end{lemma}
\begin{proof}
    Factors of homogeneous elements are homogeneous; the only partitions of a vector in $\{0,1\}^d$ are into elements of $\{0,1\}^d$.
\end{proof}

\begin{defn}
    A tensor $T$ has \textit{partition rank} one if $T = AB$ for some nonconstant elements $A,B\in R_n^{(d)}$.
    The partition rank of a tensor $T$ is the minimal $r$ for which $T$ is a sum of $r$ partition-rank-one tensors, and we write $\prk(T) = r$.
\end{defn}

For nonempty proper subset $I\subseteq [d]$, we define the map
\[
\alpha_{I; V}: V^{\otimes I}\oplus V^{\otimes [d] \setminus I}\to V^{\otimes d} \text{ by } (A, B)\mapsto AB.
\]
For $r$ tuples of nonempty proper sets $I_1, \dots, I_r\subseteq[d]$, we similarly define
\[
\alpha_{I_1,\dots, I_r; V}: \bigoplus_{s=1}^r V^{\otimes I_s}\oplus V^{\otimes [d]\setminus I_s}\to V^{\otimes d} \text{ by } (A_1,B_1,\dots, A_r,B_r)\mapsto \alpha_{I_1}(A_1,B_1)+\dots +\alpha_{I_r}(A_r,B_r).
\]
Clearly, the union of the images of $\alpha_{I_1, \dots, I_r;V}$ over all choices of $I_1, \dots, I_r\subseteq[d]$ is the locus of tensors of partition rank at most $r$.

\subsection{Partition rank at most two is closed}

In this section we prove that the locus of tensors of partition rank at most $2$ is closed.
We need first a result from commutative algebra.
Let $k$ be a field. Let $R\coloneqq k(\!(t)\!)[x_i^{(j)}]$ with the $\Z^d$-grading $\deg(x_i^{(j)}) = e_j$ and let $R_0 = k\llbracket t\rrbracket[x_i^{(j)}] \subseteq R$.

\begin{lemma}\label{lemma:explicit-prk-2-calc}
    Suppose $A_1, A_2, B_1, B_2\in R$ are tensors such that $\widetilde{T}\coloneqq A_1B_1 + A_2B_2 \in R_0$ and $\deg(\widetilde{T}) = e_1+\dots+e_d$.
    Let $T$ denote the image of $\widetilde{T}$ in $R_0/tR_0$.
    Then, we have $\prk(T)\leq 2$.
\end{lemma}

Before proving the Lemma~\ref{lemma:explicit-prk-2-calc}, we use it in our proof of the main statement of this section.

\begin{theorem}\label{thm:prk_2_closed}
    Let $V$ be a finite-dimensional vector space over an algebraically closed field $k$.
    The locus of forms of partition rank at most 2 is a closed subset of $V^{\otimes d}$.
\end{theorem}
\begin{proof}

The union of the images of $\alpha_{I_1,I_2}$ over all pairs of nonempty proper subsets partitions $I_1,I_2\subseteq [d]$ cuts out the subset of tensors of partition rank at most $2$.
Let $T$ be a closed point in the closure of this union, and suppose that $T\in \overline{\text{im}(\alpha_{I_1,I_2})}$ for some $I_1, I_2$.
We want to show that $T$ has partition rank $2$, implying that $T$ sits in the image of $\alpha_{I_3, I_4}$ for some suitable $I_3, I_4 \subset [d]$.

By the version \cite[Proposition 1]{LEHMKUHL19891} of the curve selection lemma, there exists a curve $C\subseteq V^{\otimes I_1}\oplus V^{\otimes [d]\setminus I_1}\oplus V^{\otimes I_2}\oplus V^{\otimes [d]\setminus I_2}$ such that $T\in \overline{\alpha_{I_1,I_2}(C)}$. By \cite[Corollary (a)]{LEHMKUHL19891}, there exists
\[(A_1,B_1,A_2,B_2)\in \left(V^{\otimes I_1}\oplus V^{\otimes [d]\setminus I_1}\oplus V^{\otimes I_2}\oplus V^{\otimes [d]\setminus I_2}\right) \otimes_k k(\!(t)\!)\]
such that $\alpha_{I_1,I_2}(A_1,B_1,A_2,B_2) \in V^{\otimes d}\otimes_k k\llbracket t\rrbracket$ and $\alpha_{I_1,I_2}(A_1,B_1,A_2,B_2)|_{t=0} = T$.
\Cref{lemma:explicit-prk-2-calc} implies $\prk(T)\leq 2$, so that $T$ sits in the image of $\alpha_{I_3, I_4}$ for some suitable $I_3, I_4$.
\end{proof}

We now proceed to the proof of Lemma~\ref{lemma:explicit-prk-2-calc}.
For clarity, we single out Claims~\ref{claimone}~and~\ref{claim:2-gen-reg-seq} that we need in its proof.
For an element $f = \sum_{n\geq d} t^nf_n\in R_0$ with $f_d \neq 0$ denote by $f^*$ the leading coefficient $f_d$ of $f$.

\begin{claim}\label{claimone}
    Let $A_1, B_1, A_2, B_2$ be as in Lemma~\ref{lemma:explicit-prk-2-calc}.
    Then, without loss of generality, we may assume that either 
    
    (a) $\text{ord}_t(A_i) = 0$ and $A_1^*,A_2^*$ are $k$-linearly independent, or that
    
    (b) $\text{ord}_t(A_1) = 0$ and $A_2 = 0$.
 \end{claim}
 \begin{proof}
     First we replace the pairs $(A_i,B_i)$ by $(t^{-\text{ord}_t(A_i)}A_i, t^{\text{ord}_t(A_i)}B_i)$ so that $\text{ord}_t(A_i) = 0$.
     Moving forward, we fix $A_1$ and iteratively replace $(B_1,A_2,B_2)$ by a sequence $(B_{1,0},A_{2,0},B_{2,0}), \allowbreak (B_{1,1},A_{2,1},\allowbreak B_{2,1}), (B_{1,2},A_{2,2},B_{2,2}),\dots$, beginning with $(B_{1,0},A_{2,0},B_{2,0})=(B_1,A_2,B_2)$. Suppose $(B_{1,0},A_{2,0},\allowbreak B_{2,0}),\dots, (B_{1,m},A_{2,m},B_{2,m})$ have been defined and $A_1^*, A_{2,m}^*$ are $k$-linearly dependent with $A_{2,m}^* = \lambda A_1^*$; we set 
     \[(B_{1,m+1},A_{2,m+1},B_{2,m+1})\coloneqq (B_{1,m} + \lambda A_{1,m},  t^{-1}(A_{2,m}-\lambda A_1), tB_{2,m}).\]
The claim is proven as soon as $A_1^*, A_{2,m}^*$ are linearly independent for some $m$. If instead $A_1^*, A_{2,m}^*$ are linearly dependent for all $m>0$, then $\widetilde{T} \in \bigcap_{m>0}(A_1, t^m) = (A_1)$, so we may take $A_2 = 0$. 
\end{proof}

\begin{claim}\label{claim:2-gen-reg-seq}
    Suppose $\xi_1, \xi_2\in R_0$ are tensors that form a regular sequence. If $t^NT\in (\xi_1, \xi_2)R_0$ for some $N\geq 0$, then $\prk(T)\leq 2$. 
\end{claim}
\begin{proof}
    As $T$ is homogeneous in the $\Z^d$-grading, there exist tensors $\beta_1, \beta_2\in R_0$ such that \begin{equation}\label{eq:b-rep}
    t^{-b}T = \xi_1\beta_1 + \xi_2\beta_2;\end{equation} we choose $b\leq 0$ maximal among all such representations. Write $\beta_i^{(0)}$ for the unique representative mod $t$ of $\beta_i$ in $S_0\coloneqq k[x_i^{(j)}]\subseteq R_0$.

    If $b < 0$, we have $\xi_1\beta_1^{(0)} + \xi_2\beta_2^{(0)}=0$; as $\xi_1,\xi_2$ form a regular sequence, we must have an equality of vectors $(\beta_1^{(0)}, \beta_2^{(0)}) = \gamma(-\xi_2, \xi_1)$ for some $\gamma\in R_0$. We may then replace $(\xi_1, \xi_2, \beta_1, \beta_2)$ by $(\xi_1, \xi_2, \beta_1+\lambda \xi_2, \beta_2-\lambda \xi_1)$. As $t\mid \beta_1 + \lambda \xi_2, \beta_2-\lambda \xi_1$, we divide through by $t$ to obtain $t^{-b-1}T\in (\xi_1, \xi_2)$, contradicting the assumed maximality of $b$. It follows that $b = 0$, so $T\in (\xi_1, \xi_2)R_0$. Contracting the ideal $(\xi_1, \xi_2)R_0$ down to $S_0$, we obtain $T\in (\xi_1, \xi_2)S_0$, hence $\prk(T) \leq 2$.
\end{proof}

\begin{proof}[Proof of Lemma~\ref{lemma:explicit-prk-2-calc}]
By Claim~\ref{claimone}, if $A_2 = 0$ we are done.
So, we can assume that $A_2 \neq 0$ and that $A_1^*, A_2^*$ are  $k$-linearly independent.
Let $b_i = \text{ord}_t(B_i)$ and $b = \min(b_1,b_2)$. If $b > 0$ then $T = 0$, so we may safely assume $b \leq 0$, hence we may treat $A_i, t^{-b}B_j$ as elements of $R_0$ and write
    \[
    t^{-b}T = \left(A_1(t^{-b}B_1) + A_2(t^{-b}B_2)\right)^*\in (A_1, A_2)^*,
    \]
    where $(A_1, A_2)^*$ denotes the ideal $\{f^*: f\in (A_1, A_2)\}$. We compute $(A_1, A_2)^*$ using Buchberger's algorithm. There are two main cases to consider. In the first case, $\codim(A_1^*, A_2^*) = 2$, in which case $(A_1, A_2)^* = (A_1^*, A_2^*).$ The intended claim $\prk(T)\leq 2$ follows from \Cref{claim:2-gen-reg-seq}.
    
     In the second case, $\codim(A_1^*, A_2^*) = 1$, in which case we let $D = \gcd(A_1^*, A_2^*)$. By \Cref{lemma:tensor-factor}, $D$ is a tensor and with $\delta\coloneqq \deg(D)$, we have $\delta\leq \min(\deg(A_1),\deg(A_2)).$ Put $\varepsilon_i = \deg(A_i) - \delta$ and $\mathbf{1} = e_1+\dots+e_d$. Computing $(A_1, A_2)^*$ using Buchberger's algorithm, we obtain generators $A_1^*, A_2^*, \sigma^*$, where $\sigma$ is chosen by maximizing $\mathrm{ord}_t(\sigma)$ among all those elements equivalent to $A_1(A^*_2/D) - A_2(A_1^*/D)$ mod $(A_1, A_2)$; we may without loss of generality assume $\sigma$ is homogeneous. We have $\deg(A_1^*) = \delta + \varepsilon_1, \deg(A_2^*) = \delta + \varepsilon_2, \deg(\sigma^*) = \delta+\varepsilon_1+\varepsilon_2$. By maximality of $\mathrm{ord}_t(\sigma)$, we have $\sigma^*\notin (A_1^*, A_2^*)$. 

     Let $G$ be a Groebner basis for $(A_1, A_2)$. We claim that all elements $g\in G\setminus\{A_1, A_2\}$ have $\deg(g)\geq \delta+\varepsilon_1+\varepsilon_2$, and for $g\neq \sigma$ we have $\deg(g)\not\leq \mathbf{1}$. We prove this claim by induction on the size of a partial Groebner basis $G'\subseteq G$ produced using Buchberger's algorithm. Assume we have $G' = \{A_1, A_2, \sigma, g_1,\dots, g_n\}$ and produce a new basis element by computing the $S$-polynomial $S(g, g')$ of two elements $g, g'$ of $G$. For any $1\leq i\leq n$, we have $\deg(S(g_i, \cdot)) \geq \deg(g_i) \not\leq \mathbf{1}$, so we need only consider the $S$-polynomials involving $A_1,A_2,\sigma$. Computing $S(A_1, A_2)$ and reducing mod $(A_1, A_2, \sigma)$ gives the zero polynomial, so it suffices to show $\deg(S(A_i, \sigma))\not\leq \mathbf{1}$. To see this, we compute 
     \begin{equation}\label{eqn:new-basis-element-degree}
          \deg(S(A_i, \sigma)) = \deg(\sigma) + [\deg(A_i)- \deg(\gcd(A_i^*, \sigma^*))].
     \end{equation}
    
     As $\sigma^*\notin (A_i^*)$, we have $0\leq \deg(\gcd(A_i^*, \sigma^*)) < \deg(A_i^*)$; let $j\in [d]$ such that $e_j\leq \deg(A_i^*)-\deg(\gcd(A_i^*, \sigma^*))$. Projecting onto the $e_j$-factor, we deduce that the quantity \eqref{eqn:new-basis-element-degree} is not bounded above by $\mathbf{1}$.

     The above argument shows that $t^{-b}T\in (A_1^*, A_2^*, \sigma^*)\subseteq (D, \sigma^*)$. If $\codim(D, \sigma^*) = 2$, \Cref{claim:2-gen-reg-seq} implies that $\prk(T)\leq 2$. Otherwise, $D, \sigma^*$ admit a common factor $f$, which must be a tensor in $S_0$. In this case, we have $T\in fR_0$, which implies $T\in fS_0$ and thus $\prk(T)\leq 1$.    
\end{proof}

\subsection{Partition rank is not closed in general}

Consider the $\mathbb{Z}_{\geq 0}^4$-graded polynomial ring $\mathcal{A} = k[x^{(1)}, \dots, x^{(4)}, y^{(1)}, \dots, y^{(4)}, p^{(1,2)},p^{(1,3)}, p^{(1,4)}, p^{(2,3)}, p^{(2,4)}, p^{(3,4)}]$ with $\deg(x^{(i)}) = \deg(y^{(i)}) = e_i$ and $\deg(p^{(i,j)}) = e_i + e_j$.
These variables are placeholders for, respectively, degree-1 and degree-2 tensors, and they can be seen as variables of the inverse-limit polynomial ring $R^{(4)}$ (see \Cref{lemma:GLgenisvariables}).
In particular, we look at nonconstant elements in $\mathcal{A}$ with $\mathbb{Z}_{\geq 0}^4$-grading in $\{0,1\}^4$ as tensors.

The candidate for showing that the partition rank is not a closed condition is the polynomial:
    \begin{equation}\label{eq:defT}
         \begin{split}
             T\coloneqq & \; p^{(3,4)}(x^{(1)}y^{(2)}-x^{(2)}y^{(1)}) + p^{(2,4)}(x^{(3)}y^{(1)}-x^{(1)}y^{(3)}) + p^{(1,4)}(x^{(2)}y^{(3)}-x^{(3)}y^{(2)}) + \\ & \; y^{(4)}(x^{(1)}p^{(2,3)} + x^{(2)}p^{(1,3)}+ x^{(3)}p^{(1,2)}) + x^{(4)}(y^{(1)}p^{(2,3)} + y^{(2)}p^{(1,3)} + y^{(3)}p^{(1,2)}).
         \end{split}
    \end{equation}

\begin{lemma}\label{lem:closure_of_3}
    The tensor $T$ defined in Equation~\ref{eq:defT} sits in the closure of the partition rank $\leq 3$ locus.
\end{lemma}
\begin{proof}
    Define the forms $A_{ij}\in \mathcal{A}[t]$ of degree $e_i + e_j$, with $\deg(t)=0$ as
    \[
A_{12}\coloneqq x^{(1)}y^{(2)} - x^{(2)}y^{(1)} + tp^{(1,2)},\quad A_{13}\coloneqq x^{(3)}y^{(1)} - x^{(1)}y^{(3)} + tp^{(1,3)}, \quad A_{23}\coloneqq x^{(2)}y^{(3)} - x^{(3)}y^{(2)} + tp^{(2,3)}.    
    \]
    Then we have
    \begin{align*}
    T= \lim_{t\to 0} \left (
    \begin{aligned}
    &A_{12}p^{(3,4)}+A_{13}p^{(2,4)}+A_{23}p^{(1,4)} \\+ &\frac{y^{(4)}}{t}(x^{(3)}A_{12}+x^{(2)}A_{13}+x^{(1)}A_{23}) + \frac{x^{(4)}}{t}(y^{(3)}A_{12}+y^{(2)}A_{13}+y^{(1)}A_{23})
    \end{aligned}
    \right ) ,
    \end{align*}
    so $T$ is a limit of tensors of partition rank at most $3$.
    Indeed, a direct computation shows that the expression in parentheses equals $T + t\left(p^{(1,2)}p^{(3,4)} + p^{(1,3)}p^{(2,4)} + p^{(1,4)}p^{(2,3)}\right)$.
\end{proof}

\begin{prop}\label{prop:prk4}
    The tensor $T$ of Equation~\ref{eq:defT} has partition rank equal to $4$ as an element of $\mathcal{A}$, i.e., $\prk_{\mathcal{A}}(T) = 4$.
\end{prop}

\begin{proof}
    We can rewrite the tensor $T$ as:
    \begin{equation}\label{eq:prk4}
    \begin{split}
        & x^{(1)} (y^{(2)} p^{(3,4)} - y^{(3)} p^{(2,4)}+ y^{(4)} p^{(2,3)}) + x^{(2)}(-y^{(1)} p^{(3,4)} + y^{(3)} p^{(1,4)}+ y^{(4)} p^{(1,3)}) + \\
        & x^{(3)}(y^{(1)} p^{(2,4)} - y^{(2)} p^{(1,4)} + y^{(4)} p^{(1,2)}) +  x^{(4)}(y^{(1)}p^{(2,3)} + y^{(2)}p^{(1,3)} + y^{(3)}p^{(1,2)}),
    \end{split}
    \end{equation}
    hence the partition rank is at most $4$; we show that $\prk(T) > 3$.
    
    Any $4$-tensor of partition rank 1 can be written as a product of a $3$-tensor and a $1$-tensor or else the product of two $2$-tensors (or in some cases both). Consequently, writing
    \[
    S\coloneqq A_1B_1 + A_2B_2 + A_3B_3,
    \]
    there are four cases for the partition types of the pairs $(A_\bullet, B_\bullet)$ respectively:
\begin{enumerate}[(a)]
        \item $(2,2),\quad (2,2),\quad  (2,2)$;
        \item $(1,3),\quad (2,2),\quad (2,2)$;
        \item $(1,3),\quad  (1,3),\quad (2,2)$;
        \item $(1,3),\quad  (1,3),\quad  (1,3)$.
    \end{enumerate}
We can rule out the case $(2,2) + (2,2) + (2,2)$. An arbitrary tensor of this form can be written as
\begin{equation*}
    \begin{split}
        S = & (\alpha^{(\lambda_1)} + \beta_{\lambda_1}p^{(\lambda_1)})(\alpha^{(\mu_1)}+\beta_{\mu_1}p^{(\mu_1)}) + (\alpha^{(\lambda_2)}+\beta_{\lambda_2}p^{(\lambda_2)})(\alpha^{(\mu_2)}+\beta_{\mu_2}p^{(\mu_2)}) + \\
           & (\alpha^{(\lambda_3)}+\beta_{\lambda_3}p^{(\lambda_3)})(\alpha^{(\mu_3)}+\beta_{\mu_3}p^{(\mu_3)}),
    \end{split}
\end{equation*}
        where each $\lambda_i,\mu_j$ is a partition of $\{1,2,3,4\}$ into equal-size parts and where $\alpha^{(i,j)} \in k[x^{(i)},y^{(i)}, \allowbreak x^{(j)},y^{(j)}]_{e_i + e_j}$ and $\beta_\bullet\in k$. As each $p^{(i,j)}$ occurs in $T$ with nonzero coefficient, each of the ${4\choose 2} = 6$ size-2 subsets of $\{1,2,3,4\}$ must occur among the $\lambda_\bullet, \mu_\bullet$, so the $\lambda_\bullet,\mu_\bullet$ are all distinct and the $\beta_\bullet$ are all nonzero. It follows that the $p^{(1,2)}p^{(3,4)}$ coefficient of $S$ is nonzero, hence $S\neq T$.

     The cases (b)-(d) can be ruled out using a numerical invariant of $S$. For each $i\neq j$, define 
     \[
     n_{ij}(A,B) = \begin{cases}
         2 &  \deg(B) = e_i+e_j\\
         1 & \deg(B)=e_i+e_j+e_k\\
         0 & \text{else}.
     \end{cases}
     \]
    Consider $T$ as a polynomial over $k[x^{(i)},y^{(i)}\mid i = 1,2,3,4]$. For each $\{i,j,k,l\}=\{1,2,3,4\}$, the coefficient on $p^{(i,j)}$ is $ \pm(x^{(k)}y^{(l)} \pm x^{(l)}y^{(k)})$, which does not factor as a product of linear forms. Consequently, in order for $S=T$, it must be the case that each $p^{(i,j)}$ occurs in a tensor $A_*B_*$ of the form $(2,2)$ or else $p^{(i,j)}$ occurs in \textit{at least two} tensors of the form $(1,3)$. In particular, if $S = T$ then
     \begin{equation}\label{eqn:coefficient-count-invariant}
     \sum_{1 \leq i < j \leq 4} \sum_{t=1}^3 n_{ij}(A_t,B_t) \geq \sum_{1\leq i < j \leq 4} 2 \geq 12.
     \end{equation}
     When $(A,B)$ is of type $(2,2)$, we have $\sum_{1\leq i<j\leq 4} n_{ij}(A,B) = 4$ and when $(A,B)$ is of type $(1,3)$, we have $\sum_{1\leq i<j\leq 4} n_{ij}(A,B) = 3$. For $S$ of types (b)-(d), we have
     \[
     \sum_{1 \leq i < j \leq 4} \sum_{t=1}^3 n_{ij}(A_t,B_t)= \sum_{t=1}^3 \sum_{1\leq i<j\leq 4} n_{ij}(A_t,B_t) \leq 11,
     \]
     contradicting \Cref{eqn:coefficient-count-invariant}, hence $S\neq T$ and thus $\prk(T) > 3$.
\end{proof}

\begin{theorem}\label{thm:prk_3_4}
    Let $x^{(i)} \coloneqq x_1^{(i)}$ and $y^{(i)} \coloneqq x_2^{(i)}$ in $R^{(4)}_{e_i}$ for $i = 1, \dots, 4$, let
    \[
    p^{(k,l)} \coloneqq \sum_{m \geq 1} x_m^{(k)}x_m^{(l)} \in R^{(4)}_{e_k+e_l} \qquad \text{for } 1 \leq k < l \leq 4,
    \]
    and let $T\in R^{(4)}_{\mathbf{1}}$ be the tensor of Equation~\eqref{eq:defT} evaluated at these elements.
    Then for $n$ sufficiently large the tensor $T_n \in R^{(4)}_n$, i.e., $T_n: (k^n)^{\times 4}\to k$, has partition rank $4$ and lies in the closure of the locus of tensors of partition rank at most $3$.
\end{theorem}

\begin{proof}
    We first check that the elements $x^{(i)}, y^{(i)}, p^{(k,l)}$ are part of a system of variables of $R^{(4)}$.
    As $I/I^2$ is $\N^4$-graded and the degrees $e_k + e_l$ are pairwise distinct, it suffices to check that $x^{(i)}, y^{(i)}$ are linearly independent, which is clear, and that $p^{(k,l)}\notin I^2$, which is \Cref{ex:variables}.

    By \Cref{lemma:GLgenisvariables}, the subring $k[x^{(i)}, y^{(i)}, p^{(k,l)}]$ of $R^{(4)}$ is a polynomial ring in these elements, so it is identified with the ring $\mathcal{A}$ defined at the beginning of this section.
    By \Cref{prop:prk4}, $\prk_{\mathcal{A}}(T) = 4 > 3$, so \Cref{cor:main-technique} gives $\prk(T_n) > 3$ for $n$ sufficiently large.
    On the other hand, applying $\pi_n$ to Equation~\eqref{eq:prk4} shows $\prk(T_n) \leq 4$ for every $n$.
    Explicitly, $T_n$ is obtained by substituting $x_1^{(i)}$, $x_2^{(i)}$, and $p_n^{(k,l)} = \sum_{m=1}^n x_m^{(k)}x_m^{(l)}$ for $x^{(i)}$, $y^{(i)}$, and $p^{(k,l)}$ in Equation~\eqref{eq:defT}.

    Finally, applying $\pi_n$ to the identity in the proof of \Cref{lem:closure_of_3}, we obtain that for every $t\in k\setminus\{0\}$ the tensor $T_n + t\left(p_n^{(1,2)}p_n^{(3,4)} + p_n^{(1,3)}p_n^{(2,4)} + p_n^{(1,4)}p_n^{(2,3)}\right)$ has partition rank at most $3$, so $T_n$ lies in the closure of the locus of tensors of partition rank at most $3$.
\end{proof}

\begin{lemma}[Suspension for Partition Rank]\label{lemma:prk-suspension}
    Let $T\in R^{(d)}_n$ with $\deg(T) = \mathbf{1}.$ Let $S = R^{(d)}_n[y,z]$ with $\deg(y) = e_1$ and $\deg(z) = \mathbf{1}-e_1$. Then $\prk(T+yz) = \prk(T)+1$. 
\end{lemma}
\begin{proof}
     The argument of \Cref{lemma:str-suspension} works mutatis mutandis, but since it's simpler in the partition rank case we will write out the details anyway. In what follows, we assume $d\geq 2$.
     
     The inequality $\prk(T+yz)\leq \prk(T) + 1$ is immediate. For the reverse inequality, suppose we have
\begin{equation}\label{eq:suspension-prk-decomp}
    T + yz = G_1H_1 + \dots +G_rH_r. 
    \end{equation}
    A tensor $U$ which involves both $y,z$ must have degree $\mathbf{1}$; as the degrees of $G_i,H_i$ are strictly positive, the variables $y,z$ do not occur together in any form $G_i$ or $H_i$. Without loss of generality, we therefore assume $G_i\in R^{(d)}_n[y], H_i\in R^{(d)}_n[z]$. Suppose $\deg(G_1)=\dots=\deg(G_s) = e_1$ and $\deg(G_{s+1}),\dots, \deg(G_r)\neq e_1$. In this case $G_{s+1},H_{s+1},\dots, G_r,H_r\in R^{(d)}_n$ and we may write
    \[
    G_i = g_i + \lambda_i y\qquad H_i = h_i + \mu_iz \qquad 1\leq i\leq s\qquad g_i,h_i\in R^{(d)}_n, \lambda_i,\mu_i\in k.
    \]
    As the variable $z$ must appear in the right-hand side, we must have $s\geq 1$ and $\mu_i$ not all zero; without loss of generality $\mu_1 = 1$, so that $H_1 = h_1 + z$.
 For any $2\leq i\leq s$ we have 
    \[
    G_1H_1 + G_iH_i = (G_1 + \mu_i G_i)(h_1 + z) + G_i(H_i - \mu_i h_1 - \mu_i z),
    \]
    so we may iteratively replace $H_2,\dots, H_s$ until $H_2,\dots, H_s\in R^{(d)}_n$. A similar replacement allows us to assume WLOG $G_2,\dots, G_r\in R^{(d)}_n$; equating the coefficients of $y$ and $z$ in both sides of \eqref{eq:suspension-prk-decomp} allows us to conclude that $G_1 = y, H_1 = z$. We may therefore rewrite
    \[
    T + yz = yz + G_2H_2 + \dots + G_rH_r;
    \] we conclude that $\prk(T)\leq r-1$.
\end{proof}

As in the case of strength, combining \Cref{lemma:prk-suspension} with the technique of \Cref{section:inverse-limits} gives non-closedness for all $r \geq 3$.

\begin{cor}\label{cor:prk-higher}
    For every $r \geq 3$ there exist an integer $n$ and a tensor in $(k^n)^{\otimes 4}$ of partition rank $r+1$ that lies in the closure of the locus of tensors of partition rank at most $r$.
\end{cor}
\begin{proof}
    We proceed by induction on $r$, the case $r = 3$ being \Cref{thm:prk_3_4}.
    Let $T \in R^{(4)}_n$ be a tensor of degree $\mathbf{1}$ and partition rank $r+1$ in the closure of the locus of partition rank at most $r$, and view $T$ as an element of $R^{(4)}$.
    Set $y \coloneqq x_{n+1}^{(1)}$ and $z \coloneqq \sum_{i \geq n+1} x_i^{(2)}x_i^{(3)}x_i^{(4)} \in R^{(4)}_{\mathbf{1} - e_1}$.
    We claim that the $x_i^{(j)}$ with $i \leq n$, together with $y$ and $z$, are part of a system of variables, and only need to show $z\notin I^2$.
    If $z \in I^2$, then $z$ is a finite sum of $s$ products of tensors of complementary degrees, so each truncation $z_m$ has partition rank at most $s$ as a $3$-way tensor.
    However, $z_m$ is a diagonal $3$-way tensor with $m - n$ nonzero entries, so its partition rank is $m-n$ by \cite[Section~2.2]{naslund} (see also \cite[Lemma~1]{tao}).

    By \Cref{lemma:GLgenisvariables}, $\mathcal{A} \coloneqq R^{(4)}_n[y, z]$ is a polynomial ring with $\deg(y) = e_1$ and $\deg(z) = \mathbf{1} - e_1$, and \Cref{lemma:prk-suspension} gives $\prk_{\mathcal{A}}(T + yz) = r+2$.
    By \Cref{cor:main-technique}, there exists $m > n$ such that $T + x_{n+1}^{(1)}z_m \in R^{(4)}_m$ has partition rank greater than $r+1$, hence equal to $r+2$.
    Finally, the map $(k^n)^{\otimes 4} \to (k^m)^{\otimes 4}$ sending $S \mapsto S + x_{n+1}^{(1)}z_m$ is continuous and sends tensors of partition rank at most $r$ to tensors of partition rank at most $r+1$, so $T + x_{n+1}^{(1)}z_m$ lies in the closure of the locus of tensors of partition rank at most $r+1$.
\end{proof}

\printbibliography
\end{document}